\documentclass{amsart}
\usepackage{amssymb}
\usepackage{adjustbox}

\usepackage{epsfig}  		
\usepackage{epic,eepic}       
\usepackage{tikz}
\usepackage{stmaryrd}

\newtheorem{theorem}{Theorem}[section]
\newtheorem{proposition}[theorem]{Proposition}
\newtheorem{lemma}[theorem]{Lemma}

\newtheorem{conj}[theorem]{Conjecture} 

\theoremstyle{definition}

\theoremstyle{remark}
\newtheorem{remark}[theorem]{Remark}

\numberwithin{equation}{section}

\newcommand{\e}{\varepsilon}

\begin{document}

\title[Fourier Coefficients of Critical Gaussian Multiplicative Chaos]{On the Fourier Coefficients of Critical Gaussian Multiplicative Chaos}
\author{Louis-Pierre Arguin}
\address{Mathematical Institute, University of Oxford/Baruch College and The Graduate Center, City University of New York}
\email{arguin@maths.ox.ac.uk}

\author{Jad Hamdan}
\address{Mathematical Institute, University of Oxford}
\email{hamdan@maths.ox.ac.uk}

\begin{abstract}
We continue the study of the Fourier coefficients of Gaussian multiplicative chaos (GMC) recently initiated by Garban and Vargas \cite{GarbanVargas}. We show that if $\{c_n\}_{n\geq 1}$ are the Fourier coefficients of critical GMC on the unit interval, then $(\log n)^{\alpha}c_n$ converges to zero in probability as $n$ tends to infinity for any $\alpha<1/4$.
\end{abstract}

\maketitle

\section{Introduction}

The theory of Gaussian multiplicative chaos (GMC) was initiated by Kahane \cite{kahane} in the 1980s and gives meaning to measures formally written as
\begin{align}\label{eq:GMC}
        \mu_\gamma(\mathrm{d}x)\,``\hspace{-2px}=\hspace{-2px}"\,e^{\gamma X(x)-\tfrac{\gamma^2}{2}\mathbb{E}X(x)^2}\mathrm{d}x.
\end{align}
Here, $X$ and is a logarithmically correlated Gaussian field on $\mathbb{R}^d$, $\gamma>0$ is an inverse temperature parameter, and $\mathrm{d}x$ denotes Lebesgue measure on some Borel set $D\subseteq \mathbb{R}^d$. Initially motivated by the need for a rigorous model for energy dissipation and intermittency in the Kolmogorov-Obukhov-Mandelbrot theory of developed turbulence \cite{mandelbrot, kolmogorov, oboukhov}, this theory has since found applications in many other fields, including finance \cite{finance}, quantum gravity \cite{LQG}, the extreme values of log-correlated fields \cite{DingBramsonZeitouniGFF,DingRoyZeitouni,BiskupLouidor,BramsonDingZeitouni} (see also \cite{RV14} and the references therein) and related problems in random matrix theory and number theory (e.g., \cite{WebbL2, WebbL1, HarperRMF, HarperDeterministic, SaksmanWebb}).

The difficulty in defining \eqref{eq:GMC} lies in the fact that such a field $X$ has a covariance kernel which is singular along the diagonal, being of the form
\[
    K(x,y)=\mathbb{E}[X(x)X(y)]=\log_+\frac{1}{|x-y|}+f(x,y),\quad \log_+(x)=\max(\log x,0)
\]
for some smooth function $f:D\times D\longrightarrow\mathbb{R}$. It therefore cannot be defined pointwise and is instead a distribution (in the sense of Schwartz) with probability one. To rigorously define its exponential, one typically regularizes $X$ to obtain a family $(X_t)_{t}$ of (pointwise defined) Gaussian fields converging to $X$ as $t\to\infty$, and defines $\mu_\gamma$ as the weak limit of 
\begin{align}\label{eq:reg_GMC}
    \mu_{\gamma,t}(\mathrm{d}x)=e^{\gamma X_t(x)-\frac{\gamma^2}{2}\mathbb{E}X_t(x)^2}\mathrm{d}x
\end{align}
as $t\to \infty$. This limit is known to be non-degenerate if $\gamma<\gamma_c=\sqrt{2d}$ \cite{RV10, BerestyckiElementary, Shamov}. At $\gamma=\gamma_c$, a non-trivial limit can still be obtained by considering the so-called \textit{derivative martingale} \cite{DRSV}
\begin{equation}\label{eq:critical_GMC}
        \mu_{\gamma_c,t}(\mathrm{d}x)=\big(\hspace{-2px}-\hspace{-2px}X_t(x)+\gamma_c\mathbb{E}[X_t(x)^2]\big)e^{\gamma_c X_t(x)-\frac{\gamma_c^2}{2}\mathbb{E}X_t(x)^2}\mathrm{d}x,
\end{equation}
or by multiplying \eqref{eq:reg_GMC} by a divergent factor known as the \textit{Seneta-Heyde} normalization \cite{Pow18, HRV18}, which depends on the regularization procedure. Both of these procedures yield the same limit up to a multiplicative constant, namely \textit{critical} GMC, which is non-atomic and supported on a set of Hausdorff measure zero \cite{PowellSurvey}. In this paper, our primary focus will be to study the asymptotic behavior of the Fourier coefficients of this measure.

\subsection{Harmonic analysis of multifractal measures} Given a measure $\mu$ on the unit circle identified with $[0,2\pi)$, one can ask whether or not its Fourier coefficients
\[
    \hat{\mu}(n)=\int_0^{2\pi} e^{-2\pi in\theta}\mathrm{d}\mu(\theta)
\]
vanish at infinity. Measures satisfying this condition are known as \textit{Rajchman measures} and form a class lying strictly between those of absolutely continuous and continuous measures (on $[0,2\pi)$). That these inclusions are proper is a consequence of Menšov's \cite{menshov} famous construction of a singular Rajchman measure, and Wiener's characterization of continuous measures \cite{Wiener} as those for which the $\hat{\mu}(n)$ vanish \textit{in mean}, respectively. Rajchman measures can alternatively be characterized by their common null sets (meaning sets to which they assign zero measure)---which are an intermediate between countable sets and sets of Lebesgue measure zero (see \cite{LyonsThesis, LyonsSurvey})---and are more generally related to the notion of sets of multiplicity and uniqueness in harmonic analysis \cite{kechris}.


In a recent work \cite{GarbanVargas}, Garban and Vargas initiated the study of the Fourier coefficients of (one-dimensional) GMC, proving in particular that this measure is Rajchman if $\gamma<\sqrt{2}$ \footnote{To be precise, they prove this for the GMC arising from the Gaussian free field on the unit circle.}. This resolved a longstanding open problem, and together with \cite{Falconer} was the first result of this kind for random multifractal measures. They also supplemented their result by a series of open problems that have since garnered some attention. 

For instance, define the \textit{Fourier dimension} of a measure $\mathrm{dim}_F(\mu)$ to be the largest $s$ for which $\hat{\mu}(n)\leq Cn^{-s/2}$ for some constant $C>0$, and let $\mu_{\gamma}$ denote GMC on the unit circle with parameter $\gamma$. Using a fourth moment calculation, Garban and Vargas proved that $\mathrm{dim}_F(\mu)\geq (1-\gamma^2/2)/2$ if $\gamma<1/\sqrt{2}$, and conjectured that
\begin{align}\label{eq:conj_fourierdimension}
    \mathrm{dim}_F(\mu_\gamma)=
\left\{
	\begin{array}{ll}
		1-\gamma^2  & \mbox{if } 0<\gamma < 1/\sqrt{2} \\
		(\sqrt{2}-\gamma)^2 & \mbox{if }  1/\sqrt{2}\leq\gamma<\sqrt{2}.
	\end{array}
\right.
\end{align}
This would imply that $\dim_F(\mu_\gamma)$ coincides with the \textit{correlation dimension} of $\mu_\gamma$, defined as the maximal $s$ for which
\[
    \iint_{[0,1]^2}\frac{\mu_\gamma(\mathrm{d}x)\mu_\gamma(\mathrm{d}y)}{|x-y|^s}<\infty \quad \text{a.s.},
\]
thereby making $\mu_\gamma$ a \textit{Salem measure}.

This conjecture \eqref{eq:conj_fourierdimension} was recently settled by Lin, Qiu and Tan \cite{LinQiuTanI, LinQiuTanII} in a series of works that develop a general framework to study the Fourier dimension of multiplicative chaos measures (including but not limited to one-dimensional GMC). The latter builds on a previous work of Chen, Han, Qiu and Wang \cite{ChenHanQiuWang}, which used the theory of \textit{vector-valued martingales} to compute the Fourier dimension of the Mandelbrot canonical cascade measure, answering a question of Mandelbrot from 1976.

By contrast, the picture for critical GMC is far murkier. It is well-known that the Fourier dimension of a measure is bounded above by its Hausdorff dimension, which implies that $\mathrm{dim}_F(\mu_{\gamma_c})=0$ when $\mu_{\gamma_c}$ is a critical GMC measure (i.e., the in-probability limit of \eqref{eq:critical_GMC} as $t\to \infty$). To the best of our knowledge, there are no further results in this direction, and the main open problem in \cite{GarbanVargas} was to determine whether or not critical GMC is Rajchman. 

Our goal is to provide a first step towards the resolution of this conjecture by proving the following. 

\subsection{Main results} In what follows, let $\mu=\mu_{\sqrt{2}}$ be the critical Gaussian multiplicative chaos measure arising from a $\star$--scale invariant field $X$ on $[0,1]$ (defined in Section \ref{sec:setup}), meaning that if $(X_t)_{t\geq 0}$ is a sequence of $\star$-scale cutoff approximations to $X$,
\[
    \mu(\mathrm{d}\theta):=\lim_{t\to\infty} \mu_t(\mathrm{d}\theta),\quad \text{ where } \mu_t(\mathrm{d}\theta):=\sqrt{t}e^{\sqrt{2}X_t(\theta)-\mathbb{E}X_t(\theta)^2}\mathrm{d}\theta.
\]
The limit is in probability in the space of Radon measures on $[0,1]$. Let $c_n$ denote the $n$-th Fourier coefficient of $\mu$. Our main result is the following.
\begin{theorem}
\label{thm:main} 
    For any $\alpha<1/4$, the sequence $((\log n)^{\alpha}c_n)_{n\geq 1}$ converges to zero in probability as $n$ tends to infinity.
\end{theorem}
\noindent  The polylogarithmic decay stems from what can loosely be described as a freezing phenomenon, which takes effect as soon as $\gamma\geq 1/\sqrt{2}$. This is explained in Section \ref{sec:method}, which also contains a heuristic explanation of why we believe the true order of magnitude of $|c_n|$ to be closer to $(\log n)^{-1}$.

Our proof is based on a second moment estimate for $c_n$ on a good event. While the approach is similar in spirit to Berestycki \cite{BerestyckiElementary} and Lacoin's \cite{Lacoin} constructions of subcritical and critical GMC, respectively, the presence of an oscillating factor in the integral defining $c_n$ here introduces significant complications. In particular, we require a more delicate notion of ``good" points than in \cite{BerestyckiElementary,Lacoin}, at which the underlying field $X_t(\theta)$ is controlled at scales $t\geq r_n$, where $r_n$ is a function of the frequency $n$ of the Fourier coefficient at hand. Doing so instead of controlling all scales makes the problem more tractable, at the cost of powers of $(\log n)$. An overview of the proof is given in Section \ref{sec:method}.


By comparison, Garban and Vargas bypass the explicit study of $c_n$ (when $\gamma>1/\sqrt{2}$) by appealing to the Riemann-Lebesgue lemma. To show that $c_n\to 0$ almost surely, they note that it suffices by the latter to show that $\mu_\gamma^{*d}$ is in $L^1$ for sufficiently large $d$ (where $^{*d}$ denotes $d$-fold convolution). Unfortunately, this approach breaks down at criticality as their argument would force one to take $d=\infty$.

\subsection{Proof elements}\label{sec:method} 
Let $X$ be a $\star$--scale invariant field as above, and consider the multiplicative chaos measure $\mu_\gamma:=\lim_{t\to \infty}
\mu_{\gamma,t}$ associated to $X$ for $\gamma\in (1/\sqrt{2},\sqrt{2}]$, where
\begin{equation}\label{eq:allchaos}
     \mu_{\gamma,t}(\mathrm{d}\theta):=\big(\sqrt{t}\cdot\mathbf{1}(\gamma=\sqrt{2})+\mathbf{1}(\gamma\neq \sqrt{2})\big)e^{\gamma X_t(\theta)-(\gamma^2/2)t}.
\end{equation}
Our goal is to estimate $\mathbb{E}[|c_{n,t}|^2]$, where $c_{n,t}:=\int e^{in\theta}{\mu_{\sqrt{2},t}(\mathrm{d}\theta)}.$

Before this, it is helpful to understand the second moment of the total mass $\int\mu_{\gamma,t}(\mathrm{d}\theta)$ for $\gamma<\sqrt{2}$, without the oscillating factor $e^{in\theta}$. Since $X_t$ is logarithmically correlated, Fubini's theorem implies that
\[
    \mathbb{E}\bigg[\bigg(\int\mu_{\gamma,t}(\mathrm{d}\theta)\bigg)^2\bigg]\approx \iint \frac{\mathrm{d}\theta_1\mathrm{d}\theta_2}{|\theta_1-\theta_2|^{\gamma^2}},
\]
which is finite if and only if $\gamma^2<1$ (the so-called ``$L^2$ phase" of GMC). For larger $\gamma$, the expectation is inflated by points $\theta$ where $X_t(\theta)$ is atypically large at many scales $t$. To circumvent this, one usually restricts the integral to ``good" points $\theta\in G$ for which $(X_t(\theta),t\geq 0)$ is well-behaved\footnote{This is the main idea in \cite{BerestyckiElementary, Lacoin}, for instance.}. At $\gamma=\sqrt{2}$, this is achieved by working on the event that $\{\forall \theta\in[0,1]:\theta\in{G}\}$, on which the bound
\begin{align}\label{eq:trivialbound}
\iint\mu_{\sqrt{2},t}(\mathrm{d}\theta_1)\mu_{\sqrt{2},t}(\mathrm{d}\theta_2)\leq\iint\mathbf{1}(\theta_1,\theta_2\in {G})\mu_{\sqrt{2},t}(\mathrm{d}\theta_1)\mu_{\sqrt{2},t}(\mathrm{d}\theta_2)
\end{align}
holds by monotonicity, and the right-hand side has finite expectation.

An evident hurdle that arises when trying to adapt this approach to 
\begin{equation}\label{eq:fouriercoeffintro}
    |c_{n,t}|^2=\iint e^{in(\theta_2-\theta_1)}\mu_{\sqrt{2},t}(\mathrm{d}\theta_1)\mu_{\sqrt{2},t}(\mathrm{d}\theta_2)
\end{equation}
is that one cannot directly bound as in \eqref{eq:trivialbound}, due to the factor $e^{in(\theta_2-\theta_1)}$. In particular, the set $G$ must be picked so as to guarantee the positivity and real-valuedness of
\begin{align}\label{eq:fourierexample}
    \iint e^{in(\theta_2-\theta_1)}\mathbf{1}\big(\theta_1,\theta_2\in {G}\big)\mu_{\sqrt{2},t}(\mathrm{d}\theta_1)\mu_{\sqrt{2},t}(\mathrm{d}\theta_2).
\end{align}
Furthermore, if we hope to get savings from the oscillations in the integral above, we need some control over the regularity of
\[ 
    \mathbb{E}\big[e^{\sqrt{2}(X_t(\theta_1)+X_t(\theta_2))}\mathbf{1}(\theta_1,\theta_2\in {G})\big]
\]
as a function of the gap $$\Delta=(\theta_2-\theta_1).$$ This further limits the choice of ${G}$ to one for which the event $\{\theta_1,\theta_2\in {G}\}$ is relatively simple, while being restrictive enough to make the expectation of \eqref{eq:fourierexample} of the correct order (at the very least, compensating for the additional $\sqrt{t}$ normalization present at criticality \eqref{eq:allchaos}). 

\bigskip 

We now explain why we believe that the order of $|c_n|$ is typically closer to $(\log n)^{-1}$. To begin with, a variant of the argument above reveals that the dominant contribution to \eqref{eq:fouriercoeffintro} comes from $\log|\Delta|^{-1}\asymp \log n$. For concreteness, consider the contribution to $|c_{n,t}|^2$ from $\Delta\in [n^{-100},n^{-1}]$, which is 
\begin{equation}\label{eq:fourierintegral}
        \asymp te^{-2t}\iint_{|\theta_2-\theta_1|\in[n^{-100},n^{-1}]} e^{\sqrt{2}(X_t(\theta_1)+X_t(\theta_2))}\mathrm{d}\theta_1\mathrm{d}\theta_2.
\end{equation}
(We ignore the $e^{in\Delta}$ factor since $n\Delta\leq 1$ here.)

Since $X_t$ is logarithmically correlated, one can show that for any $\theta_1,\theta_2$
\[ 
    X_s(\theta_1)\approx X_{s}(\theta_2),\quad \forall s\leq r(\Delta),
\]
with high probability
(see, e.g., Lemma 16 in \cite{DRSV}), where $r(\Delta):=\log|\Delta|^{-1}$ is the branching time. It follows that \eqref{eq:fourierintegral} is
\[
    \approx te^{-2t}\iint e^{2\sqrt{2}X_{r(\Delta)}(\theta_1)} e^{\sqrt{2}(X_t(\theta_1)-X_{r(\Delta)}(\theta_1))}e^{\sqrt{2}(X_t(\theta_2)-X_{r(\Delta)}(\theta_2))}\mathrm{d}\theta_1\mathrm{d}\theta_2.
\]
Furthermore, $e^{\sqrt{2}(X_t(\theta_i)-X_{r(\Delta)}(\theta_i))}$ will typically be of the order of $e^{t-r(\Delta)}/\sqrt{t-r(\Delta)}$, which will allow us to bound the integral above by
\begin{equation}\label{eq:almosthere}
    \int_{n^{-100}}^{n^{-1}} e^{-2r(\Delta)}\bigg(\int_0^1 e^{2\sqrt{2}X_{r(\Delta)}(\theta_1)} \mathrm{d}\theta_1\bigg)\mathrm{d}\Delta
\end{equation}
as $t\to \infty$ (noting the change of variables and that $r(\Delta)$ is bounded in $t$). 

For large $n$, we recognize the integral over $\theta_1$ as nothing but the (approximate, unnormalized) total mass of a \textit{supercritical} chaos measure, this time of parameter $\tilde{\gamma}=2\sqrt{2}$.\footnote{The emergence of such a chaos was also found by Garban and Vargas (see Theorem 1.3 in \cite{GarbanVargas}), where it is \textit{subcritical} since they only consider $2\gamma<\sqrt{2}$.} Recalling that $r(\Delta)=\log |\Delta|^{-1}$, this is known to have order 
\[
    (\log |\Delta|^{-1})^{-(3\tilde{\gamma}/(2\sqrt{2}))}e^{r(\Delta)(\tilde{\gamma}^2/2-(\tilde{\gamma}/\sqrt{2}-1)^2)}=e^{3r(\Delta)}(\log |\Delta|^{-1})^{-3}
\]
(see Theorem 2.2 in \cite{glassy}). Inserting this into \eqref{eq:almosthere},  we conclude that when $t\to \infty$,
\[
|c_{n,t}|^2\leq C\int_{n^{-100}}^{n^{-1}}\frac{1}{(\log |\Delta|^{-1})^{3}}\frac{\mathrm{d\Delta}}{\Delta} = C\int_{\log n}^{100\log n}\frac{\mathrm{d}u}{u^{3}} =O\big( (\log n)^{-2}\big)
\]
with high probability, for some constant $C>0$. In particular, this suggests a non-trivial contribution from scales $\Delta\approx n^{-a}$, even when $a>0$ is large (see Remark \ref{remark:leading_order}). 

\subsection{Conjectures for subcritical chaos}

While we do not attempt to make the heuristics from the previous section rigorous for $\gamma\in(1/\sqrt{2},\sqrt{2})$, they can nonetheless be used to make predictions in that range. There, the analogue of the integral in \eqref{eq:almosthere} contains the total mass of a chaos with parameter $\tilde{\gamma}=2\gamma$ and is therefore of the order of
\[
     \int_{n^{-100}}^{n^{-1}}e^{-\gamma^2r(\Delta)}\bigg(\frac{e^{r(\Delta)(2\gamma^2-(\sqrt{2}\gamma-1)^2)}}{(\log |\Delta|^{-1})^{3\gamma/\sqrt{2}}}\bigg)\mathrm{d}\Delta.
\]
By contrast with the critical case, the mass in this integral is concentrated around $\Delta\approx n^{-1}$, and we can factor out the correct normalization via the substitution $u=n\Delta$. What remains is approximately the integral of $1/|u|^{\kappa}$ over $[0,1]$ for some $\kappa<1$, which is integrable at zero.
\begin{conj} Let $\gamma \in (1/\sqrt{2},\sqrt{2})$, $\mu_\gamma$ denote the Gaussian multiplicative chaos measure on $[0,1]$ with parameter $\gamma$, and $c_n$ its $n$-th Fourier coefficient. Then 
\[
    (\log n)^{3\gamma/(2\sqrt{2})} n^{(\sqrt{2}-\gamma)^2/2}|c_n|
\]
converges in distribution to a (non-trivial) limiting random variable as $n\to\infty$.
\end{conj}
\begin{remark}
    We recognize the Fourier dimension of $\mu_\gamma$ in the exponent of $n$ here.
\end{remark}
Lastly, when $\gamma=1/\sqrt{2}$, the chaos of parameter $2\gamma=\sqrt{2}$ is critical and the heuristic above yields the following prediction.
\begin{conj} Let $\mu_\gamma$ denote the Gaussian multiplicative chaos measure on $[0,1]$ with parameter $\gamma=1/\sqrt{2}$, and $c_n$ its $n$-th Fourier coefficient. Then 
\[
    (\log n)^{1/4} n^{1/4}|c_n|
\]
converges in distribution to a (non-trivial) limiting random variable as $n\to\infty$.
\end{conj}
\begin{remark}
    As explained in \cite{GarbanVargas}, Section 1.2.3, $n^{1/4}c_n$ is expected to behave like the $n$-th Fourier coefficient of the Holomorphic Multiplicative Chaos (HMC), which is known to be of order $(\log n)^{-1/4}$ \cite{SoundZaman}. 
\end{remark}

\subsection{Organisation and notation} The paper is organised as follows. Section 2 contains various preliminary results on logarithmically correlated fields needed for the proof of Theorem \ref{thm:main}, which is then proved in Sections \ref{sec:main} and \ref{sec:4}.

We will use standard asymptotic notation, writing $f(T) = o(g(T))$ to mean that $|f(T)/g(T)|\to_{T\to\infty} 0$ and $f(T) = O(g(T))$ or $f(T) \ll g(T)$ to mean that
$\limsup_{T\to\infty} |f (T )/g(T )|$ is bounded. A subscripted parameter next to $o$, $O$ or $\ll$ means that the implicit constant is allowed to depend on said parameter. 

We also use 
\[
    (a\land b):=\min(a,b),\quad (a\lor b)=\max(a,b),
\]
and adopt the convention $\log x=\infty$ when $x<0$ so that, for instance, $\log(x)\land 1=1$ in this case.
\subsection{Acknowledgements} We would like to thank Zehua He for their careful reading of the manuscript, and in particular for pointing out an issue with a previous version of this work.

\section{Preliminary estimates}\label{sec:setup}
\subsection{Properties of $\star$--scale invariant fields.} We will work with chaos measures associated to a $\star$--scale invariant field $X$ on $[0,1]$ (viewed as being $1$--periodic) with $\star$--scale cut--off approximations $(X_t;t\geq0)$ (see \cite{PowellSurvey}). Our setup is essentially the same as the one in \cite{DRSV} for $d=1$, and we refer to \cite{allezrhodesvargas} for the construction of an approximating sequence $X_t$ satisfying the properties below.

\bigskip 

In what follows, $(X_t)_{t\geq 0}$ will denote a family of coupled, centered Gaussian fields on $[0,1]$ with covariances given by
\begin{align}\label{eq:*scale}
  \forall \theta_1<\theta_2,\quad\mathbb{E}[X_t(\theta_1)X_s(\theta_2)]=K_{s\land t}(\theta_2-\theta_1) :=\int_1^{e^{s\land t}}\frac{k(u(\theta_2-\theta_1))}{u}\mathrm{d}u, 
\end{align}
for some fixed \textit{seed function} $k\in C^2(\mathbb{R})$ associated to $X$, where $s\land t=\min(s,t)$. For convenience, we will assume that $k(0)=1$, $k$ is even and that it vanishes outside of the unit ball around the origin. The following properties then hold by construction:
\begin{enumerate}
    \item $\mathbb{E}[X_t(\theta)^2]=t$ for every $\theta\in [0,1]$. 
    \item For every $\theta\in [0,1]$, $(X_t(\theta);t\geq 0)$ has independent increments, making $t\mapsto X_t(\theta)$ a standard Brownian motion $B_t$ starting from zero for every $\theta$.
    \item $(X_t(\theta_1)-X_{t_0}(\theta_1);t\geq t_0)$ is independent from $(X_t(\theta_2)-X_{t_0}(\theta_2);t\geq t_0)$ for any $\theta_1,\theta_2$ satisfying $|\theta_1-\theta_2|\geq e^{-t_0}$,
\end{enumerate}  
An exact description for the joint law of $(X_t(\theta_1),X_t(\theta_2))_{t}$ can be found in \cite{DRSV}, Lemma 16, which roughly states that $X_k(\theta_1)$ and $X_k(\theta_2)$ are almost perfectly correlated up to time $k=\log 1/|\theta_1-\theta_2|$ (after which they evolve independently, as per property (3) above). For our purposes, we will only need the following estimate reflecting this correlation structure.
\begin{lemma}Let $\|k'\|_\infty=\sup_{x\in[-1,1]} |k'(x)|$. Fix $s>0$. Then uniformly in $\Delta\in [-1,1]\setminus\{0\}$,
\begin{align}\label{eq:Kestimate1}
    K_{s}(\Delta)=s\land (\log|\Delta|^{-1})+O(\|k'\|_{\infty}).
\end{align}

\end{lemma}
\begin{proof} Using the fact that $k$ is supported on $[-1,1]$ and $k(0)=1$,
    \begin{equation}\label{eq:kexpansion}
        K_s(\Delta)=\int_{1}^{e^s\land |\Delta|^{-1}}\frac{k(u|\Delta|)}{u}\mathrm{d}u= \int_1^{e^s\land |\Delta|^{-1}}\frac{\mathrm{d}u}{u}+\int_{1}^{e^s\land |\Delta|^{-1}}\frac{k(u|\Delta|)-k(0)}{u}\mathrm{d}u.
    \end{equation}
    The first term is equal to the main term in \eqref{eq:Kestimate1}. By the fundamental theorem of calculus, 
    \[
        |k(u|\Delta|)-k(0)|\leq \int_0^{u|\Delta|}|k'(x)|\mathrm{d}x \leq \|k'\|_\infty u|\Delta|
    \]
    for any $u\in [1, e^s\land |\Delta|^{-1}]$. The second term in \eqref{eq:kexpansion} is therefore $O(\|k'\|_\infty)$ by the triangle inequality.
\end{proof}
For any $t,r>0$, we define the \textit{shifted} field
\begin{equation}\label{eq:shifted}
    \big(X_t^{(r)}(\theta)\big)_{\theta\in [0,1]}:=\big((X_{t+r}-X_r)(\theta)\big)_{\theta\in [0,1]},
\end{equation}
noting that $(X_t^{(r)}(\theta))_{\theta\in [0,1]}$ is $\star-$scale invariant with kernel $k^{(r)}(x):=k(e^{r}x)$ since
\[
    \mathbb{E}[X_s^{(r)}(\theta_1)X_s^{(r)}(\theta_2)]=\bigg(\int_1^{e^{s+r}}-\int_{1}^{e^{s}}\bigg)\frac{k(u(x-y))}{u}\mathrm{d}u=\int_1^{e^{s}}\frac{k^{(r)}(u(x-y))}{u}\mathrm{d}u.
\]
In particular, $t\mapsto X_t^{(r)}(\theta)$ is a standard Brownian motion for any $r>0$ and $\theta\in [0,1]$.

Our argument will require some control on the large values of $(X_{t}(\theta))_{\theta\in [0,1]}$ at multiple scales $t$. For fixed, large $t$, the typical size of $\max_{\theta\in [0,1]}X_t(\theta)$ is known to fluctuate around
\begin{align}\label{eq:m(t)}
    m(t):=\sqrt{2}t-\frac{3}{2\sqrt{2}}\log t
\end{align}
(see \cite{DRSV, Madaule}), which is what one expects from any log--correlated field (see, e.g. \cite{ArguinSurvey}). In fact, one can typically go further and show that the {paths} $s\mapsto X_{s}(\theta), s\leq t$, for which $X_t(\theta)$ is close to the maximum cannot fluctuate too wildly. Instead, they display approximately linear growth while remaining under a suitably chosen barrier function. This phenomenon was first leveraged by Bramson \cite{Bramson} to study in the maximal displacement of branching Brownian motion, and underpins much of the subsequent work on extrema of log--correlated fields \cite{BramsonDingZeitouni, DingBramsonZeitouniGFF, Aidekon, DingRoyZeitouni}. 

For our purposes, we will only require the following upper bound of this type, due to Lacoin (\cite{Lacoin}, Proposition 2.6).

\begin{proposition}\label{prop:lacoin} Let $(X_t)_{t\geq 0}$ be defined as in Equation \eqref{eq:*scale} and 
\begin{align}\label{eq:U(t)}
    U(t):= \begin{cases}
        \sqrt{2}t-\frac{\log t}{2\sqrt{2}}+\frac{4\log\log t}{\sqrt{2}} &\quad \text{if } s\geq  e, \\
         \infty &\quad \text{if } s<e.
     \end{cases}
\end{align}
Then as $A\to\infty$,
\[
    \mathbb{P}\bigg(\exists{t\geq 0},\exists{\theta\in [0,1]}, X_t(\theta)>U(t)+A\bigg)\to 0.
\]
\end{proposition}
\begin{remark}
While \cite{Lacoin} assumes that the seed function $k$ is in $C^\infty(\mathbb{R})$, the proof of Proposition~2.6 therein only requires that $k \in C^2(\mathbb{R})$.
\end{remark}

\subsection{Brownian estimates} We record the following estimates for the probability that a Brownian motion stays under a barrier starting at time $e$ (in light of \eqref{eq:U(t)}). Such events will arise from restricting to samples for which the event in Proposition \ref{prop:lacoin} holds. 

\begin{lemma}
    Let $B_t$ be a standard linear Brownian motion. Then there exists a constant $C>0$ such that for any $t>e$, $a>0$,
    \begin{align}\label{eq:Bsup}
        \mathbb{P}\Big(\sup_{s\in [e,t]}B_s\leq a\Big) \leq C\frac{(a+1)}{\sqrt{t}+1}.
    \end{align}
    Furthermore, if $a,b>0$ and $\mathbb{P}_{(0,a)}^{(t,b)}$ denotes the law of a Brownian bridge of length $t>e$ from $a$ to $b$, then
    \begin{align}\label{eq:ballot}
        \mathbb{P}_{(0,a)}^{(t,b)}\Big(B_s\geq 0,\forall s\in [e,t]\Big) \leq C \frac{(a+1)(b+1)}{t+1}.
    \end{align}
\end{lemma}
\begin{proof}
    Both inequalities are a direct consequence of the reflection principle after integrating over the value of $B_e$.
\end{proof}

\section{Proof of the main theorem} \label{sec:main}

Let $X$ be a $\star-$scale invariant field with approximations $(X_t)_{t\geq 0}$, and define its associated (critical) GMC measure $\mu$ as
\[
    \mu(\mathrm{d}\theta):=\lim_{t\to\infty} \mu_t(\mathrm{d}\theta),\quad \text{ where } \mu_t(\mathrm{d}\theta):=\sqrt{t}e^{\sqrt{2}X_t(\theta)-\mathbb{E}X_t(\theta)^2}\mathrm{d}\theta.
\]
This limit holds in probability, and it follows that 
\begin{equation}\label{eq:c_nconvergence}
    c_{n,t}=\int_0^1 e^{in\theta}\mu_t(\mathrm{d}\theta)\overunderset{\mathbb{P}}{t\to\infty}{\longrightarrow} c_n=\int_0^1 e^{in\theta} \mu(\mathrm{d}\theta).
\end{equation}

We want to show that $\{(\log n)^{\alpha}|c_n|\}_{n\geq 1}$ converges to zero in probability for any $\alpha<1/4$, which we do by proving an upper bound of the correct order for the expectation of 
\begin{equation}\label{eq:c_nt}
    |c_{n,t}|^2=\iint_{[0,1]^2}e^{in(\theta_2-\theta_1)}\mu_t(\mathrm{d}\theta_1)\mu_t(\mathrm{d}\theta_2)
\end{equation}on a good event ${E}_{n,t}{(A)}$. This is the content of Proposition \ref{prop:second_moment} below. As a byproduct of its proof, we identify the set of pairs $(\theta_1,\theta_2)\in[0,1]^2$ over which the integral in \eqref{eq:c_nt} yields a leading order contribution, and those for which the contribution is negligible in the limit (see Remark \ref{remark:leading_order}). 

\begin{proposition}\label{prop:second_moment} Let $A>0$ and $0<\delta<1/4$. Let $ r_n:=\delta\log n$.
For any $t\geq r_n$, consider the event
\[
    {E}_{n,t}{(A)}:=\big\{\forall \theta\in [0,1], \,\theta\in G_{n,t}{(A)}\big\},
\]
where 
\begin{align}\label{eq:goodset}
G_{n,t}{(A)}=\left\{ \theta \in [0,1]: \,\,
    X_s(\theta)\leq U(s)+A,\,\forall s\in [r_n,t] \right\}
\end{align}
and $U(t)$ is defined in Equation \eqref{eq:U(t)}.
Then for any $\e>0$, there exists an integer $N>2$ such that for any $n\geq N$,
\[
  \limsup_{t\to\infty}\mathbb{E}\Big[|c_{n,t}|^2;{{E}_{n,t}{(A)}}\Big]\ll_A (\log n)^{-1/2+\e},
\]
where the implicit constant is uniform in $n$.
\end{proposition}
\noindent The choice of good event ${E}_{n,t}{(A)}$ will be explained in Section \ref{sec:4}. For the time being, we show how Theorem \ref{thm:main} follows straightforwardly from this proposition.
\begin{proof}[Proof of Theorem \ref{thm:main} assuming Proposition \ref{prop:second_moment}] 

Fix $\e>0$ and $\alpha<1/4$. By \eqref{eq:c_nconvergence} and the Portmanteau theorem,
\begin{align*}
    \mathbb{P}\Big(|c_n| ({\log n})^{\alpha}>\e\Big)&\leq \liminf_{t\to\infty}\mathbb{P}\Big(|c_{n,t}| ({\log n})^{\alpha}>\e\Big)
\end{align*}
for any $n>0$. By definition of ${E}_{n,t}{(A)}$, the pre-limit quantity on the right-hand side is bounded above by
\begin{align}\label{eq:split}
    \mathbb{P}\Big(\mathbf{1}_{{E}_{n,t}{(A)}}&|c_{n,t}| ({\log n})^{\alpha}>\e\Big)+\mathbb{P}\Big(\exists s\in [0,\infty), \exists \theta \in [0,1]: X_s(\theta)>U(s)+A\Big),
\end{align}
for any $A>0$. By Proposition \ref{prop:lacoin}, we can make the second term in \eqref{eq:split} arbitrarily small in a way that is uniform in $n$. We have that
\begin{align}
&\sup_{n>2}\mathbb{P}\Big(\exists s\in [0,\infty), \exists \theta \in [0,1]: X_s(\theta)>U(s)+A\Big) <\e/2, 
\end{align}
for sufficiently large $A$, and it follows that for such an $A$ and any $n>0$,
\begin{align*}
    \mathbb{P}\Big(|c_n| ({\log n})^{\alpha}>\e\Big)&\leq \limsup_{t\to\infty}\mathbb{P}\Big(\mathbf{1}_{{E}_{n,t}{(A)}}|c_{n,t}| ({\log n})^{\alpha}>\e\Big)+\e/2.
\end{align*}
Using Chebyshev's inequality, we can then write
 \begin{align*}
\limsup_{t\to\infty}\mathbb{P}\Big(\mathbf{1}_{{E}_{n,t}{(A)}}|c_{n,t}|({\log n})^{\alpha}>\e\Big)\leq \frac{(\log n)^{2\alpha}}{\e^2}\limsup_{t\to\infty}{\mathbb{E}\big[|c_{n,t}|^2;{E}_{n,t}{(A)}\big]},
 \end{align*}
 and the claim follows by Proposition \ref{prop:second_moment} upon taking $n\to\infty$.
\end{proof}

What's left is to prove Proposition \ref{prop:second_moment}, which is the main task here.
\section{Proof of Proposition \ref{prop:second_moment}}\label{sec:4}
\subsection{Initial reductions} To bound $|c_{n,t}|\mathbf{1}_{{E}_{n,t}{(A)}}$, we first want to extract a local condition (meaning a condition on each pair of points $\theta_1,\theta_2$) in \eqref{eq:c_nt} from the global event ${E}_{n,t}{(A)}$, noting that one cannot directly upper bound the integrand due to the presence of the oscillating factor $e^{in(\theta_2-\theta_1)}$. We instead note the following pointwise inequality on the sample space
\begin{align}\label{eq:indicatorinside}
\mathbf{1}_{{E}_{n,t}{(A)}}|c_{n,t}|^2&=\mathbf{1}_{{E}_{n,t}{(A)}}\iint_{[0,1]^2}e^{in(\theta_2-\theta_1)}\mathbf{1}\Big\{\theta_1,\theta_2\in G_{n,t}{(A)}\Big\}\mu_t(\mathrm{d}\theta_1)\mu_t(\mathrm{d}\theta_2)\\
&\leq \iint_{[0,1]^2}e^{in(\theta_2-\theta_1)}\mathbf{1}\Big\{\theta_1,\theta_2\in G_{n,t}{(A)}\Big\}\mu_t(\mathrm{d}\theta_1)\mu_t(\mathrm{d}\theta_2)\nonumber.
\end{align}
We used the fact that the indicator in the integrand is redundant in the first line and positivity of the double integral on the second, noting that it equals the modulus squared of the $n$--th Fourier coefficient of 
\[
\tilde{\mu}_t(\mathrm{d}\theta):=\mathbf{1}\Big\{\theta\in G_{n,t}{(A)}\Big\}\mu_t(\mathrm{d}\theta).
\]
It therefore suffices to prove the analogous claim (to Proposition \ref{prop:second_moment}) for the Fourier coefficients of $\tilde{\mu}_t$.
\bigskip 

To that end, fix $0<\delta<1/4$ and partition $[0,1]^2$ into the following two ranges:
\begin{align}
    \mathrm{(I)}&=\{(\theta_1,\theta_2):|\theta_2-\theta_1|\in[\Delta_n,1]\}\\
    \mathrm{(II)}&=\{(\theta_1,\theta_2):|\theta_2-\theta_1|\in[0,\Delta_n)\}, 
\end{align}
where
\[
    \Delta_n=:en^{-\delta},\text{ noting that } \log \Delta_n^{-1}=r_n-1.
\]
Denote the contribution to \eqref{eq:c_nt} coming from each of these regions by
\begin{equation}\label{eq:regions}
    \mathcal{C}_{n,t}(\#):=\iint_{(\mathrm{\#})}e^{in(\theta_2-\theta_1)}\mathbf{1}\Big\{\theta_1,\theta_2\in G_{n,t}{(A)}\Big\}\mu_t(\mathrm{d}\theta_1)\mu_t(\mathrm{d}\theta_2)
\end{equation}
for $\#\in \{\mathrm{I},\mathrm{II}\}$. 

The restriction on $\delta$ is technical, and is necessary for our bound for the contribution from (I) (see \eqref{eq:IBPLAST}) to be of the correct order. On (II), we will ignore the $e^{in(\theta_2-\theta_1)}$ by the triangle inequality. While this may seem excessively inefficient---since $\delta<1$, $e^{in(\theta_2-\theta_1)}$ will still oscillate with $n$---it will become apparent that any $\delta>0$ yields the same bound (due to the observation in Remark \ref{remark:leading_order}).

\subsection{The oscillatory range ($|\theta_2-\theta_1|\geq \Delta_n$)}
We first study the contribution from $(\mathrm{I)}$ where $e^{in(\theta_2-\theta_1)}$ oscillates very rapidly. 
 \begin{proposition}\label{prop:oscillating}
         Let $A>0$. Then
    \begin{align}\label{eq:subleading}
    \lim_{t\to\infty} \mathbb{E}\big[\mathcal{C}_{n,t}(\mathrm{I})\big]=o_A\big((\log n)^{-2}\big)
\end{align}
for large enough $n$.
\end{proposition}
\begin{proof}
Let $\Delta=(\theta_2-\theta_1)$. By the change of variables $(\theta, \Delta)=(\theta_1,\theta_2-\theta_1)$ and Fubini's theorem,
\begin{align*}
    \mathbb{E}[\mathcal{C}_{n,t}(\mathrm{I})]&\ll te^{-2t}\iint_{(\mathrm{I})} e^{in(\theta_2-\theta_1)}\mathbb{E}\big[e^{\sqrt{2}(X_t(\theta_1)+X_t(\theta_2))}\mathbf{1}\big\{\theta_1,\theta_2, G_{n,t}{(A)}\big\}\big] \mathrm{d}\theta_1 \mathrm{d}\theta_2\\
    &= te^{-2t}\int_{\Delta_n}^1 e^{in\Delta} \mathbb{E}\big[e^{\sqrt{2}(X_t(0)+X_t(\Delta))}\mathbf{1}\big\{0,\Delta\in G_{n,t}{(A)}\big\}\big]\mathrm{d}\Delta,
\end{align*}
where in the second line, we used the translation invariance of the law of $X_t$ in the following sense
\[
    \big(X_s(\theta), X_s(\theta+\Delta)\big)_{s\leq t}\overset{d}{=} \big(X_s(0), X_s(\Delta)\big)_{s\leq t}.
\]
The next step will be to exploit the cancellation coming from $e^{in\Delta}$ using integration by parts, for which we first need to transform the expectation in the integrand into a tractable function of $\Delta$. We aim to leverage the fact that since $G_{n,t}{(A)}$ only restricts $X_s(0)$ and $X_s(\Delta)$ at times $s\geq r_n$, it should, in principle, only depend on $(X_s(0), X_s(\Delta))_{s\leq r_n}$ through the value of $(X_{r_n}(0), X_{r_n}(\Delta))$. Furthermore, since $\log |\Delta|^{-1}\leq r_n-1$ when $|\Delta|\geq \Delta_n$, the integrand's dependence on $\Delta$ should only figure in the joint density of $(X_{r_n}(0), X_{r_n}(\Delta))$.

We make this precise under a tilted measure $\mathbb{Q}_\Delta$, defined through 
\begin{align}\label{eq:Girsanov}
    \frac{\mathrm{d}\mathbb{Q}_\Delta}{\mathrm{d}\mathbb{P}}=e^{\sqrt{2}(X_{r_n}(0)+X_{r_n}(\Delta))-\mathbb{E}[(X_{r_n}(0)+X_{r_n}(\Delta))^2]}.
\end{align}
Noting that $$\mathbb{E}[(X_{r_n}(0)+X_{r_n}(\Delta))^2]=2r_n+2K_{r_n}(|\Delta|),$$
(for $K_{r_n}$ in \eqref{eq:*scale}) we can write 
\begin{align*}
    &te^{-2t}\int_{\Delta_n}^1 e^{in\Delta} \mathbb{E}\big[e^{\sqrt{2}(X_t(0)+X_t(\Delta))}\mathbf{1}\big\{0,\Delta\in G_{n,t}{(A)}\big\}\big]\mathrm{d}\Delta\\
    &\ll te^{-2t+2r_n} \int_{\Delta_n}^1 e^{in\Delta+2K_t(\Delta)} \mathbb{E}_{\mathbb{{Q}}_\Delta}\big[e^{\sqrt{2}(X_{t-r_n}^{(r_n)}(0)+X_{t-r_n}^{(r_n)}(\Delta))}\mathbf{1}\big\{0,\Delta\in G_{n,t}{(A)}\big\}\big] \mathrm{d}\Delta.
\end{align*}
By properties (2) and (3) in Section \ref{sec:setup}, we know that for $\Delta\geq \Delta_n$, the processes 
\begin{equation}\label{eq:processes}
    X_{s}^{(r_n)}(0) \text{ and } X_{s}^{(r_n)}(\Delta) \text{ for } {0\leq s\leq t-r_n}
\end{equation} 
are independent, identically distributed standard Brownian motions. They are also independent from the $\sigma$--algebra generated by the field $X_s$ up to time $s\leq r_n$, and their distribution is thus unchanged under $\mathbb{Q}_\Delta$. By Girsanov's theorem, the mean of $X_{r_n}(0)$ under $\mathbb{Q}_\Delta$ is shifted by
\begin{equation}\label{eq:girsanovoscillatory}
        \sqrt{2}\cdot\text{Cov}\Big(X_{r_n}(0), X_{r_n}(0)+X_{r_n}(\Delta)\Big)=\sqrt{2}\big(r_n+K(\Delta)\big)
\end{equation}
where $K=K_{r_n}$, and the same is true for $X_{r_n}(\Delta)$ since $K(\Delta)=K(|\Delta|)$. 

Using the decomposition
\begin{align*}
    \mathbf{1}\big\{0,\Delta\in G_{n,t}{(A)}\big\}&=\mathbf{1}\big\{X_{r_n}(0), X_{r_n}(\Delta)\leq U(r_n)+A\big\}\\
    &\times \mathbf{1}\big\{X_{s}^{(r_n)}(0)\leq  U(s+r_n)+A-X_{r_n}(0),\forall s\in (0,t-r_n]\big\}\\&\times\mathbf{1}\big\{X_s^{(r_n)}(\Delta)\leq  U(s+r_n)+A-X_{r_n}(\Delta),\forall s\in (0,t-r_n]\big\}
\end{align*}
and the law of total expectation, we therefore have
\begin{align*}
    &\mathbb{E}_{\mathbb{{Q}}_\Delta}\Big[e^{\sqrt{2}(X_{t-r_n}^{(r_n)}(0)+X_{t-r_n}^{(r_n)}(\Delta))}\mathbf{1}\big\{0,\Delta\in G_{n,t}{(A)}\big\}\Big]\\
    & =\mathbb{E}_{\mathbb{{Q}}_\Delta}\bigg[\mathbf{1}\big\{X_{r_n}(0), X_{r_n}(\Delta)\leq U(r_n)+A\big\}\times\\&\mathbb{E}\Big[e^{\sqrt{2}X_{t-r_n}^{(r_n)}(0)}\mathbf{1}\big\{X_s^{(r_n)}(0)\leq U(s+r_n)+A-X_{r_n}(0), \forall s\in (0,t-r_n]\big\}\Big|\,X_{r_n}(0)\Big]\\& \mathbb{E}\Big[e^{\sqrt{2}X_{t-r_n}^{(r_n)}(\Delta)}\mathbf{1}\big\{X_s^{(r_n)}(\Delta)\leq U(s+r_n)+A-X_{r_n}(\Delta),\forall s\in (0,t-r_n]\big\}\Big|\,X_{r_n}(\Delta)\Big]\bigg].
\end{align*}
Performing the outer integral on positions $(X_{r_n}(0),X_{r_n}(\Delta))$ then shows that the above equals
\[
    f(\Delta):=\iint_{(-\infty,U(r_n)+A-\sqrt{2}(r_n+K(\Delta))]^2}p_\Delta(u_1,u_2) \beta(u_1)\beta(u_2)\mathrm{d}u_1\mathrm{d}u_2,
\]
where 
\[
p_\Delta(u_1,u_2)
=
\frac{1}{2\pi\sqrt{r_n^2-K(\Delta)^2}}
\exp\!\bigg(\!\!
-\frac{1}{2}\frac{r_n(u_1^2+u_2^2)-2K(\Delta)u_1u_2}
{\bigl(r_n^2-K(\Delta)^2\bigr)}
\bigg)
\]
is the density of a pair $(X_1,X_2)$ of centered (real) Gaussian random variables satisfying
\[
    \mathbb{E}[X_1X_2]=K(\Delta),\quad \mathbb{E}[|X_1|^2]= \mathbb{E}[|X_2|^2]=r_n,
\]
and 
\[
    \beta(u):=\mathbb{E}\Big[e^{\sqrt{2}B_{t-r_n}}\mathbf{1}\big\{B_s\leq U(s+r_n)+A-u,\forall s\in (0,t-r_n]\big\}\Big],
\]
for $(B_t)_{t\geq 0}$ a standard Brownian motion. As intended, the dependence on $\Delta$ is now solely contained in $p_\Delta$ and in the bounds of integration.

We now turn to our main task, and show that
\begin{equation}\label{eq:claim}
    \limsup_{t\to\infty} te^{-2t+2r_n}\int_{\Delta_n}^1 e^{in\Delta+2K(\Delta)}f(\Delta)\mathrm{d}\Delta =o_A\big((\log n)^{-2}\big).
\end{equation}
Integration by parts shows that the left-hand side equals
\begin{equation}\label{eq:IBP}
     \limsup_{t\to\infty}\,te^{-2t+2r_n}\bigg(\frac{e^{in\Delta+2K(\Delta)}f(\Delta)}{in}\bigg|^{1}_{\Delta_n}-\int_{\Delta_n}^1\frac{e^{in\Delta}\partial_\Delta\big( e^{2K(\Delta)}f(\Delta)\big)}{in}\mathrm{d}\Delta\bigg).
\end{equation}
We first handle the boundary term by deriving a uniform bound on $f(\Delta)$. Under the change of measure 
\[
    \frac{\mathrm{d}\tilde{\mathbb{P}}}{\mathrm{d}\mathbb{P}}=e^{\sqrt{2}B_{t-r_n}-(t-r_n)},
\]
the mean of $(B_s)_{s\leq t-r_n}$ is shifted by $\sqrt{2}s$ at time $s$ by Girsanov's theorem, and it follows that
\begin{align}\label{eq:betabound}
    \beta(u)&=e^{t-r_n} \tilde{\mathbb{P}}\big(B_s\leq U(s+r_n)+A-u,\forall s\in(0,t-r_n]\big)\nonumber\\
    &\leq e^{t-r_n}\mathbb{P}\big(B_s\leq A+\sqrt{2}r_n-u,\forall s\in (e,t-r_n]\big)\nonumber\\
    &\ll \big(A+r_n+|u|\big)\frac{e^{t-r_n}}{\sqrt{t-r_n}}
\end{align}
for any $u\in \mathbb{R}$, by Equation \eqref{eq:Bsup} and the fact that $U(s)\leq \sqrt{2}s$ for $s>e$. The implicit constant in \eqref{eq:betabound} is uniform over $u, t$ and $n$. Using this pointwise bound and the Cauchy-Schwarz inequality, we deduce that
\begin{align}\label{eq:boundf}
    f(\Delta)&\leq \frac{e^{2(t-r_n)}}{t-r_n}\mathbb{E}\Big[\big(A+r_n+|X_1|\big)\big(A+r_n+|X_2|\big)\Big] \ll_A \frac{r_n^2e^{2(t-r_n)}}{t-r_n}
\end{align}
for any $\Delta\in [\Delta_n,1]$, and in turn that
\begin{equation}\label{eq:boundarytermIBP}
\limsup_{t\to\infty}\,te^{-2t+2r_n} \frac{e^{in\Delta+2K(\Delta)}f(\Delta)}{in}\bigg|^{1}_{\Delta_n} \ll_A \frac{r_n^2 e^{2K(\Delta_n)}}{n} \ll_A \frac{r_n^2e^{2r_n}}{n}.
\end{equation}
Since $r_n=\delta\log n$ and $\delta<1/4$, this is clearly $o_A((\log n)^{-2})$.

To handle the second term inside the limit in \eqref{eq:IBP}, we first note that it is
\begin{align}\label{eq:IBP2}
    &\ll_A te^{-2t+2r_n}\frac{1}{n}\Big(\sup_{\Delta\in [\Delta_n,1]} f(\Delta)|K'(\Delta)|e^{2K(\Delta)}+e^{2K(\Delta)} f'(\Delta)\Big)\nonumber \\
    &\ll_A \frac{e^{2K(\Delta_n)}}{n}\bigg(r_n^2\cdot\sup_{\Delta\in [\Delta_n,1]} |K'(\Delta)|+te^{-2t+2r_n}\cdot\sup_{\Delta\in [\Delta_n,1]} f'(\Delta)\bigg)
\end{align}
for large enough $t$, by the triangle inequality and the uniform bound for $f(\Delta)$ in \eqref{eq:boundf}. Using the fact that $e^{2K(\Delta_n)}\leq n^{2\delta}$ and that $K'$ is uniformly bounded by
\begin{equation}\label{eq:Kbound}
        |K'(\Delta)|\leq\int_1^{e^{r_n}}|{k'(u\Delta)}|\mathrm{d}u\leq e^{r_n}\|k'\|_{\infty}\ll n^{\delta},
\end{equation}
on $[\Delta_n,1]$ (cf.~ Equation \eqref{eq:*scale}), we conclude that \eqref{eq:IBP2} is
\begin{align}\label{eq:IBP_secondbound}
    &\ll r_n^2 n^{3\delta-1}+n^{2\delta-1}(te^{-2t+2r_n})\sup_{\Delta\in[\Delta_n,1]}f'(\Delta)\nonumber\\
    &\ll o\big((\log n)^{-2}\big)+n^{2\delta-1}(te^{-2t+2r_n})\sup_{\Delta\in[\Delta_n,1]}f'(\Delta).
\end{align}
It remains to bound $f'(\Delta)$ uniformly over $[\Delta_n,1]$. For simplicity, let
\[
    H(\Delta)=U(r_n)-\sqrt{2}(r_n+K(\Delta))+A
\]
and recall that
\[
    f(\Delta)=\iint_{(-\infty,H(\Delta)]^2} p_\Delta(u_1,u_2) \beta(u_1)\beta(u_2)\mathrm{d}u_1\mathrm{d}u_2.
\]
Differentiating $f$ by repeated applications of Leibniz’s rule yields
\begin{align}\label{eq:derivative}
    f'(\Delta)&=\iint_{(-\infty,H(\Delta)]^2}\partial_\Delta p_\Delta(u_1,u_2)\beta(u_1)\beta(u_2)\mathrm{d}u_1\mathrm{d}u_2\nonumber\\\nonumber
    &+H'(\Delta)\beta\big(H(\Delta)\big)\bigg(\int_{-\infty}^{H(\Delta)}p_\Delta\big(u_1,H(\Delta)\big)\beta(u_1)\mathrm{d}u_1+\\&\hspace{120px}\int_{-\infty}^{H(\Delta)}p_\Delta\big(H(\Delta),u_2\big)\beta(u_2)\mathrm{d}u_2\bigg).
\end{align}
The boundary terms are equal by symmetry. Using \eqref{eq:betabound}, we have 
\begin{align*}
    &\int_{-\infty}^{H(\Delta)} p_\Delta\big(u_1,H(\Delta)\big)\beta(u_1)\mathrm{d}u_1 
    \\ &\ll \frac{e^{t-r_n}}{\sqrt{t-r_n}} 
   \frac{e^{-H(\Delta)^2/2r_n}}{\sqrt{2\pi r_n}}\mathbb{E}\big[(A+|X_1|+r_n)\big|X_2=H(\Delta)\big],
\end{align*}
which, when combined with the following estimates,
\begin{align}
        &|H(\Delta)|\ll_Ar_n,\\
        &\mathbb{E}\big[\big(A+|X_1|+r_n\big)\big|X_2=H(\Delta)\big] \leq \big(|H(\Delta)|+r_n\big)\ll_A r_n,\\
        &\beta\big(H(\Delta)\big)\ll_A \frac{e^{t-r_n}}{\sqrt{t-r_n}}|H(\Delta)| \ll_A \frac{r_ne^{t-r_n}}{\sqrt{t-r_n}},\\
       & H'(\Delta)=\sqrt{2}K'(\Delta)\ll n^{\delta}, \quad \text{(cf. Equation \eqref{eq:Kbound}}),
\end{align}
allows us to conclude that $f'(\Delta)$ is
\begin{align*}
    &\ll_A \frac{e^{2(t-r_n)}}{t-r_n}\bigg(r_n^2 n^{\delta}+\iint_{(-\infty,H(\Delta)]^2}\partial_\Delta p_\Delta(u_1,u_2)\beta(u_1)\beta(u_2)\mathrm{d}u_1\mathrm{d}u_2\bigg)\\
    &\ll_A \frac{e^{2(t-r_n)}}{t-r_n}\bigg(r_n^2 n^{\delta}+\iint_{\mathbb{R}^2}|\partial_\Delta p_\Delta(u_1,u_2)|(A+r_n+|u_1|)(A+r_n+|u_2|)\mathrm{d}u_1\mathrm{d}u_2\bigg).
\end{align*}
Differentiating the bivariate Gaussian density shows that the integral above is bounded by
    \begin{align*}\label{eq:p_diff}
        &\ll \frac{|K'(\Delta)\|K(\Delta)|}{(r_n^2-K(\Delta)^2)}\\&\times\mathbb{E}\bigg[\Big(A+r_n+|X_1|\Big)\Big(A+r_n+|X_2|\Big)\Big(1+\frac{1}{(r_n^2-K(\Delta)^2)}+|X_1\|X_2|\Big)\bigg],
    \end{align*}
which is
\[
    \ll |K'(\Delta)\|K(\Delta)|\mathbb{E}\Big[\big(A+r_n+|X_1|\big)\big(A+r_n+|X_2|\big)\big(1+|X_1X_2|\big)\Big]
\]
when $r_n\geq 1$, since
\[
    r_n^2-K(\Delta)^2\geq r_n^2-(r_n-1)^2\geq 1,\quad \text{ when $|\Delta|>en^{-\delta}=e^{-r_n+1}$}.
\]
Noting that $|K'(\Delta)K(\Delta)|\leq r_n n^{\delta}$ by Equation \eqref{eq:Kbound} and that
\[
    \mathbb{E}\Big[\big(A+r_n+|X_1|\big)\big(A+r_n+|X_2|\big)\big(1+|X_1X_2|\big)\Big] \ll_A r_n^3
\]
by the Cauchy-Schwarz inequality, it follows that
\[
    f'(\Delta) \ll_A \frac{e^{2(t-r_n)}}{t-r_n}\big(r_n^2n^\delta +r_n^4 n^\delta\big),
\]
and finally that
\begin{equation}\label{eq:IBPLAST}
    \limsup_{t\to\infty}n^{2\delta-1} (te^{-2t+2r_n})\sup_{\Delta\in [\Delta_n,1]}f'(\Delta) \ll_A n^{3\delta-1} r_n^4 =o_A\big((\log n)^{-2}\big).    
\end{equation}
The bound in \eqref{eq:claim} then follows from Equations \eqref{eq:boundarytermIBP}, \eqref{eq:IBP_secondbound}, and \eqref{eq:IBPLAST}.

\end{proof}

\subsection{Dominant contribution ($|\theta_2-\theta_1|\leq \Delta_{n}$)} Having shown that $\mathbb{E}\mathcal{C}_{n,t}(\mathrm{I})$ can be discarded, it remains to bound the contribution coming from $\mathbb{E}\mathcal{C}_{n,t}(\mathrm{II})$. We prove the following bound.

\begin{proposition}\label{prop:leading_order} Let $\mathcal{C}_{n,t}(\mathrm{II})$ be as in \eqref{eq:regions}. Then for any $\e>0$, 
\[
  \limsup_{t\to\infty}\mathbb{E}\big[\mathcal{C}_{n,t}\mathrm{(II)}\big]=O_A\big((\log n)^{-1/2+\e}\big)
\]
for large enough $n>0$.
\end{proposition}

We begin by introducing the auxiliary function
\begin{align}\label{eq:def_Ft}
    \mathcal{F}_{n,t}(\Delta)=te^{-2t}\cdot \mathbb{E}\Big[\mathbf{1}\big\{0,\Delta\in G_{n,t}{(A)}\big\}e^{\sqrt{2}(X_t(0)+X_t(\Delta))}\Big]
\end{align}
noting that Fubini's theorem and the triangle inequality yield
\begin{align}\label{eq:Fourier_Ft}
\mathbb{E}\big[\mathcal{C}_{n,t}(\mathrm{II})\big]\ll \int_{0}^{\Delta_n}\mathcal{F}_{n,t}(\Delta)\mathrm{d}\Delta.
\end{align}

It therefore suffices to prove that the right-hand side is $\ll (\log n)^{-1/2+\e}$ as $t\to\infty$.
The bulk of the work lies in proving the following uniform upper bound for $\mathcal{F}_{n,t}$ whose proof is deferred to Section \ref{sec:proof_of_main}. 

To state it, we introduce the \textit{branching time} 
$${r(\Delta)}=t\land \big(\log |\Delta|^{-1}\big).$$

\begin{proposition}\label{prop:main}
    Let $\mathcal{F}_{n,t}$ be defined as in Equation \eqref{eq:def_Ft}. Then there exists a constant $C(A)=C>0$ such that for large enough $t,n>0$,  $\mathcal{F}_{n,t}(\Delta)$ satisfies the bound
    \begin{equation}\label{eq:macro}
        \mathcal{F}_{n,t}(\Delta)\leq C\sqrt{r_n}\cdot\frac{e^{r(\Delta)}}{r(\Delta)} \cdot\frac{t}{(t-{r(\Delta)})+1}\cdot\frac{\log\big(r(\Delta)\big)^4}{(r(\Delta)-r_n)\lor 1}
    \end{equation}
    uniformly in $\Delta\in [0,\Delta_n]$.
\end{proposition}

\begin{proof}[Proof of Proposition \ref{prop:leading_order} assuming Proposition \ref{prop:main}]

For large enough $n>0$, we can use Proposition \ref{prop:main} to bound the integral in \eqref{eq:Fourier_Ft}, beginning with the range $[0,e^{-t}]$ where
\[
    \limsup_{t\to\infty}\int_0^{e^{-t}}\mathcal{F}_{n,t}(\Delta)\mathrm{d}t\ll_{A} \sqrt{r_n}\lim_{t\to \infty} t e^{t-t} \frac{1}{t} \bigg(\frac{\log(t)^4}{t-r_n}\bigg)=0.
\]
On $[e^{-t},\Delta_n]$, Proposition \ref{prop:main} yields
\begin{align*}
    \int_{e^{-t}}^{\Delta_n}\mathcal{F}_{n,t}(\Delta)\mathrm{d}\Delta\ll_A \sqrt{r_n}\int_{e^{-t}}^{\Delta_n} \frac{1}{r(\Delta)}\frac{t}{(t-{r(\Delta)})+1}\frac{\log(r(\Delta))^4}{(r(\Delta)-r_n)\lor 1}\frac{1}{\Delta}\mathrm{d}\Delta.
\end{align*}
We first take care of the range where $r(\Delta)$ is close to $r_n$. On $[\Delta_n/e^3,\Delta_n]$, say, we have
\begin{align*}
    &\limsup_{t\to\infty} \int_{\Delta_n/e^3}^{\Delta_n}\mathcal{F}_{n,t}(\Delta)\mathrm{d}\Delta\\&\ll\limsup_{t\to\infty} \frac{\sqrt{r_n}\log(r_n+2)^4}{r_n-1}\frac{t}{t-r_n+2} \int_{\Delta_n/e^3}^{\Delta_n} \frac{\mathrm{d}\Delta}{\Delta} \\&\ll\frac{(\log r_n)^4}{\sqrt{r_n}} \int_{r_n-1}^{r_n+2} \mathrm{d}\log(\Delta^{-1})\ll\frac{(\log r_n)^4}{\sqrt{r_n}}.
\end{align*}
Using the substitution $u=r(\Delta)=\log \Delta^{-1}$, we split the remaining range into $u\in [r_n+2,t/2)$ and $u\in [t/2,t)$. The contribution from the latter is
\begin{align*}
    &\ll \sqrt{r_n}\int_{t/2}^{t} \frac{1}{u}\frac{t}{(t-u)+1}\frac{\log(u)^4}{(u-r_n)\lor 1}\mathrm{d}u \\
    &\ll  \sqrt{r_n}\frac{t}{t-1}\frac{\log(t)^4}{(t-1-r_n)} +\sqrt{r_n}t\log(t)^4\int_{t/2}^{t-1}\frac{1}{u(u-r_n)(t-u)}\mathrm{d}u \\
    &\ll \sqrt{r_n}\frac{\log(t)^{5}}{t} \to 0 \text{ as }t\to \infty.
\end{align*}
On the former, we have $t/(t-u)\leq 2$ and we conclude that 
\begin{align*}
   \limsup_{t\to\infty}\int_{e^{-t}}^{\Delta_n}\mathcal{F}_{n,t}(\Delta)\mathrm{d}\Delta &\ll_A \sqrt{r_n} \int_{r_n+2}^{\infty} \frac{1}{u}\frac{\log(u)^4}{(u-r_n)}\mathrm{d}u\\&\ll_A \frac{\log (2r_n)^4}{\sqrt{r_n}}\int_2^{r_n} \frac{\mathrm{d}x}{x}+\sqrt{r_n}\int_{r_n}^\infty \frac{\log (2x)^4}{x^2}\mathrm{d}x\\
   &\ll_A \frac{(\log \log n)^{5}}{\sqrt{\log n}},
\end{align*}
which is $\ll_A(\log n)^{-1/2+\e}$ for any $\e>0$ as claimed.
\end{proof}
\begin{remark}\label{remark:leading_order}
    For any $a>b>\delta$, the same argument shows that 
    \[
        \lim_{t\to\infty} \int_{n^{-a}}^{n^{-b}} e^{in\Delta}\mathcal{F}_{n,t}(\Delta)\mathrm{d}\Delta\ll_A \frac{1}{(\log n)^{1/2-\e}}
    \]
    whereas for any $s_1(n)\leq s_2(n)$ both tending to $+\infty$ in $n$, 
    \[
        \lim_{t\to\infty} \int_{n^{-s_2(n)}}^{n^{-s_1(n)}} e^{in\Delta}\mathcal{F}_{n,t}(\Delta)\mathrm{d}\Delta =o_A\bigg(\frac{1}{(\log n)^{1/2-\e}}\bigg).
    \]
    This suggests that the main contribution to $\mathbb{E}[\mathcal{C}_{n,t}\mathrm{(II)}]$ (and in turn $\mathbb{E}[|c_{n,t}|^2;{E}_{n,t}{(A)}]$) comes from pairs $(\theta_1,\theta_2)$ for which $|\theta_2-\theta_1|\approx n^{-a}$ for \textit{all} scales $a>\delta$. 
\end{remark}




\section{Upper bound for $\mathcal{F}_{n,t}(\Delta)$} \label{sec:proof_of_main} 
This section proves Proposition \ref{prop:main}. Recalling that
\begin{equation*}
\mathcal{F}_{n,t}(\Delta)=te^{-2t}\mathbb{E}\big[e^{\sqrt{2}(X_{t}(0)+X_{t}(\Delta))}\mathbf{1}\big\{X_{s}(*)\leq U(s)+A,\forall s\in [r_n,t], *\in \{0,\Delta\}\big\}\Big]
\end{equation*}
an application of Girsanov's theorem shows that $\mathcal{F}_{n,t}(\Delta)$ equals
\begin{align}\label{eq:F_tRHS}
        te^{2K_{t}(\Delta)}\mathbb{P}\Big(X_{s}(*)\leq A+{U}(s)-\sqrt{2}(s+K_{s}(\Delta)),{\forall s\in [r_n,t], *\in \{0,\Delta\}}\Big).
\end{align}
The goal will be to prove that this quantity satisfies the bound
\[
    \ll_{A} \sqrt{r_n} \frac{e^{r(\Delta)}}{r(\Delta)}\frac{t}{(t-r(\Delta))+1}\cdot\frac{\log (r(\Delta)-r_n)^{4}}{(r(\Delta)-r_n)\lor 1},
\]
by arguing similarly to the proof of Proposition 2.6 in \cite{Lacoin}. We require a slightly sharper estimate than the one proved there, which we obtain by utilizing the lower order terms in $U(s)$. 



Assume for now that $r(\Delta)-r_n>1$. Observe first that by symmetry, we can add the restriction $X_{r(\Delta)}(0)\leq X_{r(\Delta)}(\Delta)$ to the event in \eqref{eq:F_tRHS} since doing so can only change its probability by a factor of $1/2$. Secondly,  we can use the estimate
\[
    K_{s}(\Delta)=s\land (\log |\Delta|^{-1})+O(\|k'\|_\infty)=s\land r(\Delta)+O(1),
\]
(cf. Equation \eqref{eq:Kestimate1}), which is uniform in $s \geq 0$ and $\Delta\in [0,\Delta_n]$. It follows that there exists a constant $C=C(A)>0$ for which the intersection of $\{X_{r(\Delta)}(0)\leq X_{r(\Delta)}(\Delta)\}$ with the event in \eqref{eq:F_tRHS} is contained in $E_1\cap E_2\cap E_3$, where
\begin{align*}
E_1 &:=\Big\{X_{s}(0)\leq C+U(s)-2\sqrt{2}s,\forall s\in [r_n,r(\Delta)]\Big\}&\\
E_2 &:=\Big\{X_{s}^{(r(\Delta))}(0)\leq C-\sqrt{2}r(\Delta)-X_{r(\Delta)}(0),\forall s\in [0,t-r(\Delta)]\Big\}\\
E_3 &:=\Big\{X_{s}^{(r(\Delta))}(\Delta)\leq C-\sqrt{2}r(\Delta)-X_{r(\Delta)}(0),\forall s\in [0,t-r(\Delta)]\Big\}
\end{align*}
Note that to get the inequality in $E_3$, we used (and subsequently dropped) the restriction $X_{r(\Delta)}(0)\leq X_{r(\Delta)}(\Delta)$.


We thus need to bound
\begin{align}\label{eq:finalprob}
    \mathbb{P}(E_1\cap E_2\cap E_3)=\mathbb{E}\Big[\mathbf{1}_{E_1}\mathbb{P}\big(E_2\cap E_3\big|\Sigma_\Delta\big)\Big],
\end{align}
where $\Sigma_{\Delta}=\sigma((X_s(\theta))_{\theta\in[0,1],s\leq r(\Delta)})$. This will require a ``linearized" version of ${U}(s)$. For $r(\Delta)\geq e$, define the \textit{slope} 
\[
    \alpha_\Delta=\alpha = \bigg(\sqrt{2}-\frac{\log r(\Delta) }{2\sqrt{2}r(\Delta)}\bigg),
\]
so that 
\begin{equation}\label{eq:linearbarrier}
      {U}(s)\leq \alpha s+\frac{4}{\sqrt{2}}{\log\log r(\Delta)}\, \text{ for all $s\in [e,r(\Delta)]$},
\end{equation}
with equality at $s=r(\Delta)$. For simplicity, we also define
\[
    V_{\Delta,A}(s)=V(s)=C+\frac{4}{\sqrt{2}}\log\log r(\Delta)+(\alpha-2\sqrt{2})s
\]
which is greater than $U(s)-2\sqrt{2}s$ for all $s\in [e,r(\Delta)]$.

We can now estimate \eqref{eq:finalprob}. Letting $(B_s)_{s\leq r(\Delta)}$ denote a standard Brownian motion, we note on the one hand that
\begin{align*}
    &\mathbb{P}\Big(E_1\,\Big|\, X_{r_n}(0)=y, X_{r(\Delta)}(0)=z\Big)\\
    &\leq \mathbb{P}\Big(X_{s-r_n}^{(r_n)}(0)\leq V(s),\forall s\in [0,r(\Delta)-r_n]\,\Big|\,X_{r_n}(0)=y,X_{r(\Delta)}(0)=z\Big)\\
    &= \mathbb{P}\big(B_s\leq V(s), \forall s\in [0,r(\Delta)-r_n]\,\big|\,B_{r_n}=y,B_{r(\Delta)}=z\big),
\end{align*}
and that the right-hand side can be estimated using Equation \eqref{eq:ballot}. Indeed, it is equal to the probability that a Brownian bridge from $(V(r_n)-y)$ to $(V(r(\Delta))-z)$ remains positive, which is  
\[
    \ll\bigg(\frac{(V(r_n)-y+1)(V(r(\Delta))-z+1)}{(r(\Delta)-r_n)+1}\land 1\bigg).
\]
On the other hand, Equation \eqref{eq:Bsup} yields 
\begin{align*}
    &\mathbb{P}\big(E_2\cap E_3\,\big|\, \Sigma_\Delta\big)\ll_A \bigg(\frac{(V(r(\Delta))-X_{r(\Delta)}(0)+1)^2}{((t-r(\Delta))+1)}\land 1\bigg).
\end{align*}



By integrating over the values of $$(y,x)=(X_{r_n}(0),X_{r(\Delta)}(0)-X_{r_n}(0))$$ we can therefore conclude that $\mathcal{F}_{n,t}(\Delta)$ is
\begin{align}\label{eq:final_estimate}
        &\ll te^{2K_{t}(\Delta)}\mathbb{E}\Big[\mathbf{1}_{E_1}\mathbb{P}\big(E_2\cap E_3\,\big|\,X_{r(\Delta)}(0)\big)\Big]\nonumber\\
        &\ll_A \frac{te^{2r(\Delta)}}{(r(\Delta)-r_n)(t-r(\Delta)+1)}\\&\times\int_{-\infty}^{V(r_n)}\frac{e^{-\frac{y^2}{2r_n}}}{\sqrt{r_n}}\int_{-\infty}^{V(r(\Delta))-y} \frac{e^{-\frac{x^2}{2(r(\Delta)-r_n)}}
    }{\sqrt{r(\Delta)-r_n}}F(y,(x+y))\mathrm{d}{x}\mathrm{d}{y}\nonumber,
\end{align}
where
\[
    F(a,b):={(V(r_n)-a+1)}{(V(r(\Delta))-b+1)^3}.
\]
The change of variables $z=x+y$ yields
\begin{align*}
    &\frac{te^{2r(\Delta)}}{(r(\Delta)-r_n)(t-r(\Delta)+1)}\\&\times\int_{-\infty}^{V(r_n)}\int_{-\infty}^{V(r(\Delta))} \frac{e^{-z^2/2r(\Delta)}}{\sqrt{r(\Delta)}}\frac{e^{-(y-\frac{r_n}{r(\Delta)}z)^2/2(\tfrac{r_n}{r(\Delta)}(r(\Delta)-r_n))}}{\sqrt{(r_n/r(\Delta))(r(\Delta)-r_n)}} F(y,z)\mathrm{d}z\mathrm{d}y.
\end{align*}
We first evaluate the integral over $y$, which equals
\[
    \mathbb{E}\Big[\mathbf{1}\big(X<V(r_n)\big)\big(V(r_n)-X+1\big)\Big], \text{ where } X\sim \mathcal{N}\Big(\mu=\frac{r_n}{r(\Delta)}z,\sigma^2=\frac{r_n(r(\Delta)-r_n)}{r(\Delta)}\Big).
\]
This can be decomposed into
\begin{align}\label{eq:DOUBLE}
    &\big(V(r_n)-\mu+1\big)\mathbb{P}(X<V(r_n))+\mathbb{E}\big[\mathbf{1}(X<V(r_n))\big(X-\mu\big)\big]\\
    &\leq \big(V(r_n)-\mu+1\big)+\sigma\mathbb{E}\Big[Z\mathbf{1}\big(Z<(V(r_n)-\mu)/\sigma\big)\Big]
\end{align}
(for $Z$ a standard Gaussian), which is equal to
\[
   (V(r_n)-\mu+1)+\frac{\sigma}{\sqrt{2\pi}} e^{-(V(r_n)-\mu)^2/(2\sigma^2)}.
\]
Inserting this into \eqref{eq:DOUBLE} yields the bound
\begin{align*}
     \ll&\frac{te^{2r(\Delta)}}{(r(\Delta)-r_n)(t-r(\Delta)+1)}\\&\times \bigg(\int_{-\infty}^{V(r(\Delta))} \frac{e^{-z^2/2r(\Delta)}}{\sqrt{r(\Delta)}}\Big(V(r_n)-\frac{r_n}{r(\Delta)}z+1\Big)\big(V(r(\Delta))-z+1\big)^3 \mathrm{d}z\\
     &\hspace{20px}+\sigma\int_{-\infty}^{V(r(\Delta))} \frac{e^{-z^2/2r(\Delta)-(V(r_n)-\frac{r_n}{r(\Delta)}z)^2/2\sigma^2}}{\sqrt{r(\Delta)}}\big(V(r(\Delta))-z+1\big)^3\mathrm{d}z\bigg)
\end{align*}
which by the substitution $\bar{z}:=V(r(\Delta))-z$ equals
\begin{align}\label{eq:lastbound}
    &\frac{te^{2r(\Delta)}}{(r(\Delta)-r_n)(t-r(\Delta)+1)}\bigg(\int_{0}^{\infty} \frac{e^{-(V(r(\Delta))-\bar{z})^2/2r(\Delta)}}{\sqrt{r(\Delta)}}\Big(\frac{r_n}{r(\Delta)}
\bar{z}+1\Big)\big(\bar{z}+1\big)^3 \mathrm{d}\bar{z}\nonumber\\
     &\hspace{50px}+\sigma\int_{0}^{\infty} \frac{e^{-(V(r(\Delta))-\bar{z})^2/2r(\Delta)-\bar{z}^2/2(r_
     \Delta-r_n)}}{\sqrt{r(\Delta)}}\big(\bar{z}+1\big)^3\mathrm{d}\bar{z}\bigg)
     \nonumber\\ 
     &\leq \frac{(1+\sigma)te^{2r(\Delta)}}{(r(\Delta)-r_n)(t-r(\Delta)+1)}\int_0^\infty \frac{e^{-(V(r(\Delta))-\bar{z})^2/2r(\Delta)}}{\sqrt{r(\Delta)}}(\bar{z}+1)^4\mathrm{d}\bar{z}
     \nonumber\\
     &\leq  \frac{(1+\sigma) te^{2r(\Delta)}}{(r(\Delta)-r_n)(t-r(\Delta)+1)} \frac{e^{-V(r(\Delta))^2/2r(\Delta)}}{\sqrt{r(\Delta)}} \int_0^{\infty} e^{\tfrac{V(r(\Delta))}{r(\Delta)}\bar{z}}(\bar{z}+1)^4\mathrm{d}\bar{z}.
\end{align}
Recalling that 
\[
    V(r(\Delta))=C+\frac{4}{\sqrt{2}}\log\log r(\Delta)-\sqrt{2}r(\Delta)-\frac{1}{2\sqrt{2}}\log r(\Delta),
\]
the remaining integral is $\ll\int_0^\infty e^{-\bar{z}}(\bar{z}+1)^4\mathrm{d}z<\infty$ for large enough $n$. Expanding the Gaussian term in \eqref{eq:lastbound} and noting that  $\sigma\leq \sqrt{r_n}$, we conclude that
\[
    \mathcal{F}_{n,t}(\Delta)\ll_A \frac{te^{r(\Delta)}}{(r(\Delta)-r_n)(t-r(\Delta)+1)} \frac{\sqrt{r_n}}{r(\Delta)}(\log r(\Delta))^4,
\]
which is the claimed bound.

When $r(\Delta)-r_n<1$, the claim follows by repeating the argument above while replacing $E_1$ by
\[
    \Big\{X_{r_n}(0)\leq C+U(r_n)-2\sqrt{2} r_n\Big\},
\]
and $r(\Delta)$ by $r_n$ in $E_2$ and $E_3$.

\appendix

\section*{Funding}
L.-P. A.'s research is supported in part by the NSF grant DMS-2153803 and the EPSRC grant EP/Z535990/1. J.H. supported by the EPSRC Centre for Doctoral Training in Mathematics of Random Systems: Analysis, Modelling and Simulation (EP/S023925/1). For the purpose of open access, the authors have applied a CC BY public copyright licence to any author accepted manuscript arising from this submission.

\bibliographystyle{abbrv}
\bibliography{biblio}

\appendix

\end{document}